\documentclass[11pt]{amsart}
\usepackage[margin=1.5in]{geometry}

\usepackage{nicefrac}
\usepackage[dvipsnames]{xcolor}
\usepackage{esvect}
\usepackage{graphicx}      
\usepackage{tikz}
\usepackage{pgfplots}
\usepackage{amsmath,amssymb, amsfonts,amsthm}
\usepackage{graphicx}
\usepackage[]{amsaddr}
\usepackage{lmodern}
\usepackage{tikz-cd}
\newcommand{\N}{\mathbb{N}}
\newcommand{\R}{\mathbb{R}}

\newcommand{\cL}{{\mathcal{L}}}
\newcommand{\cW}{{\mathcal{W}}}

\usepackage{microtype}

\newcommand{\Span}{\operatorname{span}}

\makeatletter
\def\blfootnote{\xdef\@thefnmark{}\@footnotetext}
\makeatother

\newcommand{\xc}[1]{\vspace{.1cm}

\noindent {\em #1} }

\newcommand{\mab}[1]{\vspace{.1cm}

\noindent {\textbf{#1 }}} 

\newtheorem{definition}{Definition}
\newtheorem{theorem}{Theorem}[section]
\newtheorem{lemma}{Lemma}[section]
\newtheorem{proposition}[theorem]{Proposition}
\newtheorem{corollary}[theorem]{Corollary}

\newtheorem*{mainth*}{Main Theorem}

\usepackage{subcaption}

\title[Super-linearization and polynomial automorphisms]{On the invariance of super-linearization under polynomial automorphisms}

\begin{document}

\author[Harshana]{Anmol Harshana}
\address[Harshana]{Department of Aerospace Engineering and Coordinated Science Laboratory, University of Illinois, Urbana-Champaign}
\email[Harshana]{anmolh2@illinois.edu}

\author[Belabbas]{Mohamed-Ali Belabbas}
\address[Belabbas]{Electrical and Computer Engineering Department and Coordinated Science Laboratory, University of Illinois, Urbana-Champaign.}
\email[Belabbas]{belabbas@illinois.edu}

\maketitle

\begin{abstract}

We prove that the super-linearizability of polynomial systems is preserved by all currently known classes of polynomial automorphisms of $\R^n$. We then establish connections between such automorphisms and a sufficient condition for super-linearizability. 
\end{abstract}

\section{Introduction}

Consider the autonomous system defined by the equation:

\begin{equation}\label{eq:mainsys}
\dot x = f(x)\mbox{ with } f: \R^n \to \R^n
\end{equation}

where $f$ is a polynomial vector field. This system is said to be {\em super-linearizable} if the trajectories of the vector field $f$ are a (linear) projection of the trajectories of a linear autonomous system, albeit one evolving in a higher-dimensional state-space. A more precise definition is provided in Definition~\ref{defn:lifted_super_linearizable}.
It is important to note that not all polynomial systems are super-linearizable. Determining checkable necessary and sufficient conditions for super-linearizability of a general dynamical system (or even a polynomial system) remains an open problem. However, a checkable {\em sufficient} condition was proposed in~\cite{belabbas2023sufficient}, which we refer to as the {\em WDG condition}. We will revisit this condition in Section~\ref{ssec:wdg}.

In this paper, we study the effect of polynomial coordinate transformations (referred to as {\em polynomial automorphisms}) on the super-linearizability of a polynomial system. More generally, we are interested in the invariance  properties of the space of super-linearizable vector fields under coordinates changes; i.e., we are interested in the question:  {\em Given a super-linearizable  vector field and an automorphism (i.e., a change of variables), does it hold that the vector field obtained after applying the automorphism is super-linearizable?}

The general answer to that question is negative.  Indeed,  consider the following example: the linear system
\begin{equation*}
\left\{ 
\begin{aligned}
    \dot x_1 &= x_2\\
    \dot x_2 &=x_1
    \end{aligned}
    \right.
\end{equation*}
obviously belongs to the class of super-linearizable  polynomial systems; applying the change of coordinates $(z_1,z_2)=(\sinh(x_1), \sinh(x_2))$, we obtain the dynamics \begin{equation}\label{eq:example_sinh}
    \left\{\begin{aligned}
    \dot z_1 = \sqrt{1+z_1^2}\sinh^{-1}{z_2}\\ 
    \dot z_2 = \sqrt{1+z_2^2}\sinh^{-1}{z_1}.
\end{aligned}\right.
\end{equation}
It can be shown that system~\eqref{eq:example_sinh} is not super-linearizable.
We provide proof of that fact in Appendix~\ref{subsec:ex1_proof}.

\vspace{.1cm}
\noindent {\bf Tame, Stably Tame and Wild Automorphisms.} We  provide a brief, high-level account of what is known about polynomial automorphisms.
By polynomial automorphism, we mean  a polynomial diffeomorphism of $\R^n$ whose inverse is also polynomial.

Polynomial automorphisms have been classified into the so-called {\em tame} and {\em wild} automorphisms~\cite{shestakov2004tame}. We provide precise definitions below but mention here that tame automorphisms are constructively defined as compositions of invertible affine transformations and {\em elementary automorphisms}, and wild automorphisms are the ones that are not tame.
It is known that if $n=2$, all automorphisms are tame~\cite{kuttykrishnan2009some, jung1942ganze}. In dimension $n=3$,  there are a {\em few} known examples of wild automorphisms, including Nagata's~\cite{shestakov2003nagata} and Anick's~\cite{umirbaev2007anick} automorphisms.

However, both of the above-mentioned wild automorphisms are {\em stably tame}, i.e., when embedded in $\R^{n'}$ for $n'>n$  in a way made precise in the following section, they are {\em tame} automorphisms; see~\cite{smith1989stably, kuttykrishnan2009some}. In fact, all the known examples of wild automorphisms in $\R^n$ have been shown to be stably tame and whether wild, non-stably tame automorphisms exist in $\R^n$ is an open question.

\vspace{.1cm}
\noindent {\bf Our Contribution.} The main contribution of this paper is to show that the class super-linearizable polynomial vector fields is closed under automorphism. Said otherwise, if a polynomial vector field is super-linearizable, then the vector field obtained after a polynomial change of variables is super-linearizable.

Our second contribution deals with the class of polynomial vector fields that meet the so-called {\em WDG condition}. As mentioned earlier, determining whether a polynomial vector field is super-linearizable is an open problem. However, there exists an efficiently checkable condition, the WDG condition described below in Section~\ref{ssec:wdg}, that characterizes a (strict) subclass of super-linearizable polynomial systems. It is then natural to ask whether this subclass is closed under polynomial automorphisms. We show that the answer is negative; however, taking inspiration from the definition of stably tame automorphisms, adding what we called below ``stabilizing'' observables, we can exhibit cases when the WDG condition is preserved.

\vspace{.1cm}
\noindent{\bf Related Works.}
Global linearization of nonlinear systems is linked to the Koopman Operator or Carleman Linearization, both of which map the nonlinear system to a linear system  in infinite dimensions. For general applications, it is necessary to find a finite approximation of the same~\cite{bevanda2021koopman} and calculate error bounds on using limited dimensions~\cite{amini2022carleman}. 

The concept of a global linearization of ordinary differential equations has also been studied in algebra under the label of polyflows (see~\cite{bass1985polynomial, van1994locally, hosler1989polynomial}), i.e., when can the solutions of nonlinear ordinary differential equations be represented as a polynomial of the initial conditions at each instant of time. The general forms of polynomial vector fields in 2 dimensions that give polyflows as solutions have been studied in this context~\cite{bass1985polynomial, hosler1989polynomial}.

More recently, Jungers and Tabuada~\cite{jungers2019non} used up to $N^{th}$ Lie derivatives to linearize a nonlinear differential equation while accounting for asymptotic stability. 
In~\cite{ko2024minimum}, Ko and Belabbas derived the minimum number of observables required for super-linearization and proposed two invariants that arise in strong super-linearizations.
In ~\cite{arathoon2023koopman}, Kvalheim and Arathoon exhibit counterexamples to the claim of super-linearization of a system being only possible when there is a single or a continuum of equilibria. In ~\cite{kvalheim2023linearizability}, the same authors provide a detailed topological study of fundamental features and limitations of applied Koopman theory.

\vspace{.1cm}
\noindent{\bf Notation and Conventions.}
We denote by ${\mathcal P}_n$  the space of polynomial vector fields of $\R^n$ in $n$ variables. The subset of ${\mathcal P}_n$ consisting of super-linearizable polynomial vector fields is denoted by ${\mathcal S}_n$. The subset of super-linearizable systems ${\mathcal S}_n$ that satisfy the WDG condition as explained in \ref{thm:wdg} is denoted by $\cW_n$ We will denote by $e^{tf}x_0$ the trajectory of $\dot x = f(x)$, with $x(0)=x_0$,  evaluated at $t$. If $f=Ax$ is a linear vector field, we simply write $e^{At}x_0$. 
We let $\Pi_n:\R^m \to \R^n$, $m \geq n$, denote the canonical projection map on  the first $n$ coordinates: namely, give the  vector $z=(z_1,\ldots, z_m)$, $\Pi(z)=(z_1,\ldots, z_n)$. We denote by
$D \phi_x$ the Jacobian of $\phi(x)$ evaluated at $x$. Given a vector field $f$ and a function $p$, we let $(\cL_f p)(x) := (D_{p(x)}f(x))$. We denote the $i^{th}$ component of vector $v$ by $[v]_i$ or $v_i$.

\section{Background and Statement of the main results}\label{sec:background}

We provide in this section the key definitions relevant to our work and then proceed to state the main results. We start with the definition of a super-linearizable system/vector field:

\begin{definition}[Lifted and Super-linearizable system]\label{defn:lifted_super_linearizable}
    Let $f:\R^n \to \R^n$ and $p:\R^n \to \R^k$ be polynomials. Then $\tilde f:\R^{n+k}\to \R^{n+k}$ is a lift of $f$ by $p$ if
$$
\Pi_n e^{t \tilde f}z_0 = e^{tf}x_0 \mbox{ for all } t \geq 0
$$
and $z_0 = [x_0,p(x_0)]^\top$. 
If $\tilde f(z) =A z$ for some $A\in\R^{(n+k)\times(n+k)}$, then $f$ is called {\em super-linearizable}. 
\end{definition}

Said otherwise,  the trajectory of the  system $\dot z = \tilde f(z)$, initialized at $(x_0,p(x_0))$, equals the trajectory of the original system initialized at $x_0$ after projection by $\Pi$. The entries of $p$ are referred to as {\em observables}.

For example, consider $f(x) = \begin{bmatrix}
    x_1, & x_2+x_1^2
\end{bmatrix}^\top $ and $g(x) = x_1^2$. Then a lifted system with $p=g(x)$ is $ \tilde f = \begin{bmatrix}
z_1, & z_2+z_3, & 2z_3
\end{bmatrix}^\top $ and $z_0 =\begin{bmatrix}x_{0,1}, &  x_{0,2}, &x_{0,1}^2\end{bmatrix}^\top$

We now provide the definitions related to the polynomial automorphisms of $\R^n$. The simplest ones are the affine transformations:
\begin{definition}[Affine Transformation]\label{defn:2}
    An {\em affine transformation} is an automorphism $\phi: \R^n \rightarrow \R^n$ of the form $\phi(x) = Ax + b$ where $b\in \R^n$ and $A\in \R^{n \times n}$ is invertible.
\end{definition}

The following class of automorphisms is the building block of the more complex classes described below:
\begin{definition}[Elementary Transformation]\label{defn:3}
    An {\em elementary transformation} is an automorphism $\phi: \R^n \rightarrow \R^n $ of the form $$\phi(x) = [x_1,x_2,\ldots x_{n-1},x_n + g(x_1,x_2,\ldots x_{n-1})]^\top$$ where $x = [x_1,x_2\ldots x_n]^\top \in \R^n $ and $g(x_1,x_2,\ldots x_{n-1})$ is a polynomial of degree greater than 1.
\end{definition}

We note that $g$ is {\em any} polynomial that depends only on $x_1,\ldots x_{n-1}$, and that it enters at the $n$th entry of $\phi$. This particular form ensures that $\phi$ is indeed invertible with a polynomial inverse and, thus, an automorphism. In fact, we can write its inverse explicitly as as follows: setting $y = \phi(x) = [x_1,x_2,\ldots x_{n-1},x_n + g(x_1,x_2,\ldots x_{n-1})]^\top $, we obtain  $$x = \phi^{-1} (y) = [y_1,y_2,\ldots,y_n - g(y_1,y_2,\ldots y_{n-1})]^\top.$$

Since the permutation matrices are affine transformations, by composing an affine transformation with an elementary automorphism, we can obtain automorphisms of the type $$(x_1,\ldots, x_n) \mapsto (x_1,\ldots,x_{i}+g(x_1,\ldots,\hat x_i,\ldots, x_n),\ldots, x_n),$$ for any $i \in \{1,\ldots,n\}$
where $\hat x_i$ indicates that this variable is omitted.

We can now define the main class of known  polynomial  automorphisms, namely the tame automorphisms: 

\begin{definition}[Tame and Wild Automorphisms]\label{defn:4}
    A {\em tame automorphism} is a finite composition of affine and elementary transformations. A wild automorphism is an automorphism that is not tame.
\end{definition}

The following definition introduces an extension to the notion of tame automorphism, which is reminiscent of the definition of super-linearization:

\begin{definition}[Stably Tame Automorphisms]\label{defn:6}
    A polynomial automorphism $\psi:\R^n \rightarrow \R^n$ is called {\em stably tame} if there exists $m>0$ and a tame automorphism $\phi:\R^{n+m} \rightarrow \R^{n+m}$ and a polynomial map $y:\R^n \to \R^m$ such that $$\Pi_n(\phi(x,y(x))) = \psi(x)\mbox{ for all } x\in \R^n.$$ The entries of $y$ are referred to as {\em stabilizing variables}.
 
\end{definition}

We can now state the main results of the paper. We recall that by transformation of a polynomial system, we mean a change of variables.

\begin{theorem}\label{thm:tame}
    The class of polynomial super-linearizable systems is closed under transformation by tame automorphisms. 
\end{theorem}
As a corollary, we will show the following
\begin{corollary}\label{cor:stablytame}
The class of polynomial super-linearizable systems is closed under transformation by stably tame automorphisms.
\end{corollary}

\subsection{The WDG condition: a sufficient condition for super-linearization}\label{ssec:wdg}
As already mentioned, exhibiting a checkable condition for super-linearization, even in the case of polynomial vector fields, is an open problem.  It is easy to see, see, e.g.~\cite{levine1986nonlinear,coomes1991linearization}, that a vector field $f:\R^n \to \R^n$ is super-linearizable {\em if and only if}  \begin{equation}\label{eq:spansuper}
\dim_\R \operatorname{span}\{f,\cL_f f, \cL_f^2 f,\cL_f^2 f,\ldots\}< \infty
\end{equation}
is finite. Note that $(\cL_f f)(x(t))=\frac{d}{dt} f(x(t))$ is the {\em total time derivative} of $f$. Hence, we can paraphrase the condition as saying that a system is super-linearizable if the iterated total-time derivatives of $f$ span a finite-dimensional vector space. Since there are no known upper bound on the number of derivatives one needs to evaluate to guarantee the span is finite-dimensional, this condition does not provide a test to decide in finite time whether $f$ is super-linearizable. 

In~\cite{belabbas2023sufficient}, the authors provide a sufficient condition for super-linearizability that throughout the paper, we will refer to this sufficient condition as the {\em WDG condition}, which stands for {\em weighted dependency graph condition}. To state it, we introduce the weighted dependency graph: given a system $\dot x = f(x)$ in $\R^n$, we associate to it a directed graph on $n$ nodes and with (weighted) edge set
$$E = \{v_iv_j \mid (D_x(f))_{ij} \neq 0\}.$$ The edge weight is then $\gamma_{v_iv_j}=(D_x(f))_{ij}$. Note that $i=j$ is permitted, i.e., self-loops are allowed.
The result is then
\begin{theorem}[WDG condition]\label{thm:wdg}
The system $\dot x = f(x)$ is super-linearizable if the product of the edge weights along all cycles of its WDG is a constant.
\end{theorem}
 Consider the following system, for example:
\begin{equation}\label{example:2}
\left\{
\begin{aligned}
    \dot x_1 &= -x_1+x_3 \\
    \dot x_2 &= 2x_1+x_3\\
    \dot x_3 &=2x_2\\
    \dot x_4 &= x_1^2+x_3^2
\end{aligned}
\right.
\end{equation}
The weighted graph of the above system is shown in Figure~\ref{fig:WDG example}. The cycles in the graph are colored green, and it can be easily seen that the edge weights in all the cycles are constant, and hence so is their product along said cycles. We conclude that the above system satisfies the WDG condition and is super-linearizable.

\begin{figure}
    \centering
    \begin{tikzpicture}[scale=.7,->, >=latex, node distance=1.2cm, thick]

    \node[fill=black, circle, inner sep=1.5pt, label=below:$x_1$] (x1) at (0,0) {};
    \node[fill=black, circle, inner sep=1.5pt, label=above:$x_2$] (x2) at (2,2) {};
    \node[fill=black, circle, inner sep=1.5pt, label=below:$x_3$] (x3) at (4,0) {};
    \node[fill=black, circle, inner sep=1.5pt, label=below:$x_4$] (x4) at (2,-2) {};

    \path[Green] (x1) edge [loop left,min distance=0.6 cm] node[above] {-1} (x1);
    \path[Green] (x1) edge node[above left] {2} (x2);
    \path[blue] (x1) edge node[below] {2$x_1$} (x4);
    \path[Green] (x2) edge [bend left=20]  node[above] {2} (x3);
    \path[Green] (x3) edge [bend left=20] node[left] {1} (x2);
    \path[blue] (x3) edge node[below,xshift=.2cm] {2$x_3$} (x4);
    \path[Green] (x3) edge node[above] {1} (x1);

\end{tikzpicture}
    \caption{A weighted dependency graph of a system that follows the sufficiency condition}
    \label{fig:WDG example}
\end{figure}
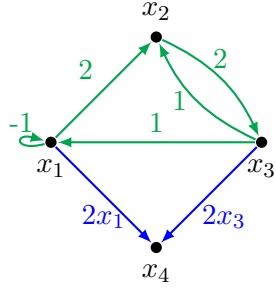

\section{Proof of the main results}
We start with some preliminary Lemmas that will be used repeatedly in the proof of the main results.

The first result states that if a polynomial system is the projection of a higher-dimensional super-linearizable system, then it is also super-linearizable. Precisely, we have

\begin{lemma}\label{lemma:1}
    Let $h \in {\mathcal S}_m$ and $f \in {\mathcal P}_n$ for $0 < n <m$. If it holds that for some polynomial $p$  
    $$
    \Pi_n (e^{th}[x_0;p(x_0)]^\top) = e^{tf} x_0, \mbox{ for all } t \geq 0, x_0 \in \R^n,
    $$
    then $f \in {\mathcal S}_n$.
\end{lemma}

\begin{proof}
    Since $\dot y = h(y)$ is super-linearizable, there exists $k >0$, $A \in \R^{(m+k) \times (m+k)}$, and a polynomial function $q:\R^m \rightarrow \R^k$ so that 
    
    $$\Pi_m(e^{At}z_0) = e^{th}y_0 \mbox{ for all } t \geq 0,$$ with $z(0)=[y_0, q(y_0)]$.
    
    From the assumption of the Lemma  and the above relation, we have   
    $$
    e^{tf}x_0=\Pi_n(e^{th}[x_0;p(x_0)]^\top) =  \Pi_n \Pi_m(e^{At}z_0)$$ for $z_0=[y_0,q(y_0)]^\top$ and $y_0= [x_0;p(x_0)]^\top$. But since $\Pi_n \circ \Pi_m=\Pi_n$ for any $0 < n \leq m$,  the above relation yields
     $$\Pi_n(e^{At}z_0) = e^{tf}x_0 \mbox{ for all } t \geq 0.$$ This shows that $\dot x = f(x)$ is super-linearizable and concludes the proof.
    \end{proof}
The above Lemma can be paraphrased by saying that if the diagram below commute for some super-linearizable $h$, then $f$ is super-linearizable.
\[
\begin{tikzcd}
\R^m \arrow[r, "e^{th}"] \arrow[d, "\Pi_n"'] & \R^m \arrow[d, "\Pi_n"] \\
\R^n \arrow[r, "e^{tf}"'] & \R^n
\end{tikzcd}
\]

The following Lemma provides in a sense a reciprocal statement, allowing us to conclude that a higher-dimensional system is super-linearizable system from the knowledge that a lower-dimensional system is super-linearizable. 

\begin{lemma}\label{lemma:2}
    Let $f\in {\mathcal S}_n$ and $g:\R^n \rightarrow \R^k$ be a polynomial, then the system $\dot y = \tilde f := [f,g]^\top$ is super-linearizable.
\end{lemma}

\begin{proof}
    It suffices to prove the result for $k=1$; indeed, if that statement holds, then $\tilde f$ is super-linearizable and thus satisfies the requirement of the Lemma, which can then be applied iteratively to prove the general case. 
    
    We thus assume that $k=1$. Since $f$ is super-linearizable, there exists $ m>0$,  $A \in \R^{(n+m)\times(n+m)}$ and a polynomial function $p:\R^n \to \R^m$ such that $\Pi_n(z(t)) = x(t)$ for $\dot z =  Az $ with $z(0) = [x_0,p(x_0)]^\top$. 

    Now set $w=g(x)$ and compute the total time derivative of $w$. Since $x(t)=\Pi_n(z(t))$, we have 
     $$\dot w=\sum_{i=1}^n\dfrac{\partial g(z_1,\ldots,z_n)}{\partial z_i}f_i(z_1,\ldots,z_n) =:h(z_1,\ldots,z_n)$$

Now set  $\tilde z := [z,w]^\top$. It obeys the (nonlinear) dynamics
        \begin{equation}\label{eqn:3}
        \frac{d}{dt} {\tilde{z}} = [\dot z, \dot w]^\top = [Az,h(z_1,\ldots,z_n)]^\top=:\tilde F(\tilde z).
    \end{equation}
    We claim that system~\eqref{eqn:3} is super-linearizable. To see that the claim holds, we prove that it follows the WDG sufficient condition of Theorem~\ref{thm:wdg}. Proceeding by contradiction, assume that there exists a cycle in the WDG with a non-constant weight on an edge. Since the only node with nonlinear edge weights is $w$ (or $\tilde z_{n+1}$), the cycle must pass through that node. However, there are no outgoing edges from the node in the WDG corresponding to the variable $w$, since $\frac{\partial \tilde F}{\partial w}=0$. Hence, no so such  cycle exist, which proves the claim.

    Now, reorder the entries of $\tilde z$ as  $u := [z_1, \ldots, z_n, w, z_{n+1}, \ldots, z_{n+k}]$ and reorder similarly the entries of $\tilde F$ in $ F$ so that $\dot u = F(u)$. Then, it holds that
    $$\Pi_{n+1} e^{t F} u_0 = e^{t\tilde f}\Pi_{n+1} u_0 \mbox{ for all } t \geq 0.
    $$ and $y_0 = [x_0,g(x_0)]^\top,u_0 = [x_0,g(x_0),p(x_0)]^\top $
    Using Lemma~\ref{lemma:1}, $\dot {y} = \tilde f(y)$  is also super-linearizable. In other words, $\tilde f = [f,g]^\top \in {\mathcal S}_{n+1}$. This proves the case $k=1$.  
\end{proof}

\subsection{Proof of Theorem~\ref{thm:tame}}
Recall that a tame automorphism can be expressed as a composition of affine and elementary automorphisms. So, to prove Theorem~\ref{thm:tame}, it is sufficient to show that super-linearization is preserved both by affine transformation and by elementary transformations. This is the approach we take to prove the result. We start by showing that super-linearization is an invariant of affine transformation.

\begin{proposition}\label{prop:affine}
    Super-linearization is preserved under affine automorphisms.
\end{proposition}
The proof is a consequence of Lemmas~\ref{lemma:3} and~\ref{lem:invariancetranslation} below.

\begin{lemma}\label{lemma:3}
    If $\dot x = f(x)$ is super-linearizable, with $x \in \R^n$, and $P \in \R^{n \times n}$ is an invertible matrix, then $\dot z =Pf(P^{-1}z)$ is super-linearizable. 
\end{lemma}
The proof for the Lemma above can be found in the appendix~\ref{subsec:linearoperatorproof}. 
\begin{lemma}\label{lem:invariancetranslation}
    If $\dot x = f(x)$ is super-linearizable, with $x \in \R^n$, $c \in \R^n$ and $z=x-c$, then  $\dot z =h(z)$ is super-linearizable. 
\end{lemma}
    The proof is elementary and we only sketch it due to space constraints.
\begin{proof}
 Let $p(x)$ be observables super-linearizing $f(x)$, then $p(z+c)$ are observables super-linearizing $h$.
\end{proof}

\begin{proposition}\label{prop:elementary}
   Let $\dot x = f(x)$ with $f\in {\mathcal S}_n$ and  $\phi$ be an elementary automorphism.  Then, the system $\dot y = h(y)$ where $y = \phi(x)$ is super-linearizable, i.e., $h\in {\mathcal S}_n$.
\end{proposition}

\begin{proof}
Let $ \phi(x) = [x_1,x_2,\ldots, x_n + g(x_1,x_2,\ldots, x_{n-1})]^\top$ be an elementary transformation, with inverse $$x = \phi^{-1}(y) = [y_1,y_2,\ldots y_n - g(y_1,y_2,\ldots, y_{n-1})]^\top. $$

 We first write the explicit form of $h(y)$.
We have
$\dot y = h(y):= D\phi_x f(x)$ with $x = \phi^{-1} (y)$.  
We can expand this Jacobian as follows 
\begin{equation}\label{eqn:elementary_jacobian}
\begin{aligned} D \phi_x &= \left[\begin{array}{c | c} 
\mathbf{I}_{n-1} & 0\\ 
  \hline 
  \begin{array}{cc}
   \dfrac{\partial g}{\partial x_1}  \cdots \dfrac{\partial g}{\partial x_{n-1}}   &  \\
  \end{array} & 1 
 \end{array}\right] \\
 &= \mathbf{I}_{n} + \left[\begin{array}{c | c} 
\mathbf{O}_{n-1} & 0\\ 
  \hline 
  \begin{array}{cc}
   \dfrac{\partial g}{\partial x_1}  \cdots \dfrac{\partial g}{\partial x_{n-1}}   &  \\
  \end{array} & 0 
 \end{array}\right].
 \end{aligned}\end{equation}

 With the above expansion of Jacobian, we have
 \begin{align}\label{eq:DFF}
 D\phi_x f(x)&= f(x) + \left[\begin{array}{c | c} 
\mathbf{O}_{n-1} & 0\\ 
  \hline 
  \begin{array}{cc}
   \dfrac{\partial g}{\partial x_1}  \cdots \dfrac{\partial g}{\partial x_{n-1}}   &  \\
  \end{array} & 0 
 \end{array}\right] f(x) \notag \\&= f(x) + [0,0,0\ldots \sum_{i=1}^{n-1}  \dfrac{\partial g}{\partial x_i}f_i(x)]^\top
  \end{align}

  Recall that under $y = \phi(x)$, $x_i = y_i, 1\leq i\leq n-1$, so we have $ \dfrac{\partial}{\partial x_i}g(x_1,\ldots,x_{n-1}) =\dfrac{\partial g}{\partial y_i}g(y_1,\ldots,y_{n-1})$.  
 Next, the last row of the right-hand side of~\eqref{eq:DFF} can be expressed as
$f_n(y_1,\ldots,y_{n-1},y_n-g(y_1, \ldots,y_{n-1}) + \sum_{i=1}^{n-1} \dfrac{\partial g}{\partial y_i} f_i(y_1,\ldots, y_n-g(y_1,\ldots,y_{n-1})).$
Together, the above equations give the explicit form of $h(y)$:
$$
h(y) = \begin{bmatrix}
    f_1(y_1,\ldots, y_{n-1},y_n-g(y_1,\ldots,y_{n-1})\\
    \vdots\\
    f_{n-1}(y_1,\ldots, y_{n-1},y_n-g(y_1,\ldots,y_{n-1})\\
    f_n + \sum_{i=1}^{n-1} \dfrac{\partial g}{\partial y_i} f_i    
\end{bmatrix}$$

We now show that $h(y)$ is super-linearizable. 
To this end, consider the dynamics for $z \in \R^{n+1}$ given by%
$$\dot z = \tilde f(z) := \begin{bmatrix}f_1(z_1,\ldots z_n)\\
\vdots
\\
f_n(z_1,\ldots,z_n)\\
f_n(z_1,\ldots z_n) + \sum_{i=1}^{n-1} \dfrac{\partial g}{\partial z_i}(z_1,\ldots,z_{n-1}) f_i\end{bmatrix}  $$
Since $f$ is super-linearizable, we known from  Lemma~\ref{lemma:2}  that $\tilde f$ is super-linearizable as well.

Now, noting that $\frac{d}{dt} (y_n - g(y_1,\ldots,y_{n-1})) = f_n$, we see that for all $y_0 \in \R^n$, if we set $z_0=[y_{0,1},\ldots,y_{0,n-1}, y_{0,n}-g(y_{0,1},y_{0,n-1}),y_{0,n}]^\top$, then it follows from the definitions of $h$ and $\tilde f$ that $z_i(t)=y_i(t)$ for $1 \leq i \leq n-1$ and $z_{n+1}(t)= y_n(t)$.
Now, let us reorder the entries of $z$ as follows: $u = [z_1,\ldots,z_{n-1},z_{n+1},z_n]$ and reorder similarly the entries of $\tilde f$, so that $\dot u = \tilde h (u)$. 

The arguments above shows that 
    $$\Pi_n e^{t \tilde h} u_0 = e^{th}y_0 \mbox{ for all } t \geq 0, y_0 \in \R^{n},
    $$
    where $u_0=[y_{0,1},\ldots,y_{0,n-1},y_{0,n},y_{0,n}-g(y_{0,1})]^\top$
    Hence, using Lemma~\ref{lemma:1}, $\dot {y} = h(y)$  is also super-linearizable.
\end{proof}

Using Proposition~\ref{prop:affine} and Proposition~\ref{prop:elementary} we can conclude that super-linearizability
is preserved under transformations by tame automorphisms.

\subsection{Proof for Corollary~\ref{cor:stablytame}: Case of Stably Tame Automorphisms}
Due to space limitations, we provide a detailed sketch of the proof.
Let $\dot x = f(x)$ be a super-linearizable system, and let $\psi:\R^n \to \R^n$ be a stably tame automorphism that becomes tame in $n+k$ dimensions when using a stabilizing variable $w = g(x)$. Let $\dot y=h(y)$ be the transformed system, where $y = \psi(x)$. 
Let $\phi:\R^{n+k}\to\R^{n+k}$ be the tame automorphism obtained by adjoining $\psi$ with $w$. Note that we can also write $\Pi_{n}(\phi(z)) = \psi(x) $ where $z=[x,w]^\top$. Now, the following steps and the diagram below illustrate the method of going from $\dot x = f(x)$ to $\dot y = h(y)$ by lifting the system first, and we will provide arguments for super-linearizability of $\dot y = h(y)$. 
\[
\begin{tikzcd}
\dot x= f(x) \arrow[r, " y=\psi(x)"] \arrow[d, "z={(x, \, w)}"'] & \dot y = h(y)  \\
\dot z = \tilde f(z) \arrow[r, "r=\phi(z)"'] & \dot r = \tilde h(r) \arrow[u, "y=\Pi(r)"']
\end{tikzcd}
\]
\begin{itemize}
    \item Lift $f$ by '$w$' variables $(w=g(x))$ to $x$: we have $z = [x,w]^\top$ and $\dot z = \tilde f(z)$. From Lemma~\ref{lemma:2}, we have that $\tilde f \in {\mathcal S}_{n+k}$.

    \item Introduce $r = \phi(z)$ and let the dynamics after applying the automorphism $\phi$ be  $\dot r = \tilde h(r)$. Because $\phi$ is {\em tame} by construction,  Theorem~\ref{thm:tame} yields that $\tilde h\in {\mathcal S}_{n+k} $. 

    \item Now, take the projection of the transformed system in the previous step. We have $\Pi_n(r(t)) = \Pi_n(\phi([x,w]^\top(t)) )= \psi(x(t)) = y(t)$  for all  $t \geq 0$. In particular, 
    $$\Pi_n(e^{\tilde h}r_0) =  e^{th}y_0 \mbox{ for all }  t \geq 0$$ where  $r_0 = \phi([x_0,g(x_0)]) = [y_0,g(\psi^{-1}(y_0))].$ (Recall that $\psi$ is an automorphism). Now, using lemma~\ref{lemma:1}, we can argue that $\dot y = h(y)$ is super-linearizable i.e. $h\in {\mathcal S}_{n} $.
\end{itemize}

\subsection{WDG and Elementary transformations}

We now investigate whether the invariance of the WDG condition under polynomial automorphisms. The invariance in this case appears to be more subtle, as we now outline, before concluding with open questions.

We first highlight that strict invariance, as proven in Theorem~\ref{thm:tame} for the case of super-linearizable systems, does not hold. To see this, consider the following system:

 $$\dot x = Ax \mbox{ with } A = \left[\begin{array}{c  c} 
0 & 1\\ 
0  & 0 
 \end{array}\right]$$ 
together with the elementary automorphism $\phi(x_1,x_2) = [x_1,x_2-x_1^2]^\top$, for which $\phi^{-1}(y) = [y_1,y_2+y_1^2]^\top$.
Setting $y = \phi(x_1,x_2)$, we have  and $\dot y = D\phi_{\phi^{-1}(y)}A\phi^{-1}y$ is given by
\begin{equation}\label{eq:wdg_counterexample}
 \begin{aligned}
 \dot y &= \left[\begin{array}{c  c} 
1 & 0\\ 
-2y_1  & 1 
 \end{array}\right]\left[\begin{array}{c  c} 
0 & 1\\ 
0  & 0 
\end{array}\right]\begin{bmatrix}y_1\\y_2+y_1^2\end{bmatrix} = \left[\begin{array}{c } 
 y_2+y_1^2\\ 
-2y_1(y_2+y_1^2)  
 \end{array}\right]\end{aligned}\end{equation}
It can be easily seen that system~\ref{eq:wdg_counterexample}  $\notin \cW_2$, but we know from Theorem~\ref{thm:tame} that it is super-linearizable ($\in {\mathcal S}_2$).

Now, we introduce an observable $p(y_1,y_2):= y_2+y_1^2$ and set $y_3=p$. We have that $\dot y_3=0$. Consider the {\em non-linear} system obtained by adding this one observable to the original dynamics:
 \begin{align*}
\left[\begin{array}{c } 
 \dot y_1\\ 
\dot y_2 \\
\dot y_3
 \end{array}\right] = \left[\begin{array}{c } 
 y_3\\ 
-2y_1y_3\\
0
 \end{array}\right].
 \end{align*}
It clearly satisfies the WDG condition. 

Hence, similar to the case of {\em stably tame} automorphism, which required the addition of stabilizing variables to be tame, we observe that a linear system requires the addition of {\em one} ``stabilizing'' observable to be WDG condition compliant after a change of variables. This is the result we prove next:

\begin{proposition}\label{prop:1}
 Consider the linear system $\dot x = Ax$, with $A = (a)_{ij} \in \R^{n\times n}$, let $\phi$ be an elementary automorphism and $\dot y = h(y)$ be the system obtained after the change of variables $y =\phi(x).$ Then
 \begin{enumerate}
     \item There exists a ``stabilizing'' observable $p:\R^n\to \R$ so that the lifted system  with $p$ is in $\cW_{n+1}$. 
    \item if $(a)_{in} =0 \forall i \neq n$, then $h \in \cW_n$.
    \end{enumerate}
     \end{proposition}
\begin{proof}

Let $\phi(x) = [x_1,x_2,\ldots x_n + g(x_1,x_2,\ldots, x_{n-1})]^\top$.
We then have
$\dot y = D\phi_x \dot x = D \phi_{\phi^{-1} (y)} A \phi^{-1}(y)$ where $D\phi_x = D \phi_{\phi^{-1} (y)}$ is the Jacobian of $\phi$ calculated at $x=\phi^{-1}(y)$.  

Using the expansion of $D\phi_x$ given in equation~\eqref{eqn:elementary_jacobian}, we can write:
$$\dot y  = A\begin{bmatrix}y_1\\ \vdots\\ y_{n-1}\\y_n-g(y_1,\ldots,y_{n-1})\end{bmatrix} + \begin{bmatrix}0\\ \vdots \\0\\ \sum_{i=1}^{n-1} \sum_{j=1}^{n} \dfrac{\partial g}{\partial y_i} A_{ij}y_j\end{bmatrix}=h(y)$$

\noindent{\em Proof of item 2.} Assume that $a_{in}=0$ for $i=1,\ldots,n$. We claim that there are no cycles in the WDG of $h(y)$ along which the product of the weights is non-constant. To see this, first note that since  $(a)_{in} =0$ for $i=1\ldots,n$, $h_i$  are linear for $i=1\ldots,n-1$, and nonlinear terms may only appear in $h_n$. 
Since $g$ does not depend on $y_n$, $\frac{\partial h_n}{\partial y_n}=0$ and there are no self-loops on $y_n$ with non-constant weight. The only edges with potentially non-constant weight are thus the edges incoming to $y_n$, and these edges cannot belong to a cycle as there are no outgoing edges from $y_n$ to $y_i, i \neq n$ in the weighted graph (because $a_{in}=0$ implies that $\frac{\partial h_i}{\partial y_n}=0$  for $i=,1\ldots,n$); this proves the claim and the second item of the Proposition.

\noindent{\em Proof of item 1.} We introduce the observable $w = y_n - g(y_1,y_2,y_3,\ldots,y_{n-1}).$ 
We set
  \begin{align*}
 \dot w := \dot y_n - \sum_{i=1}^{n-1} \dfrac{\partial g}{\partial y_i} \dot y_i  =  [A\phi^{-1}(y)]_n,
 \end{align*}
then, the dynamics of $y$ can be written as $$
\dot y  = A[y_1,\ldots, y_{n-1},w)]^\top + [0,0,0\ldots, \sum_{i=1}^{n-1} \sum_{j=1}^{n} \dfrac{\partial g}{\partial y_i} A_{ij}y_j]^\top$$
provided that $w_0=p(y_0)$. This shows that 
setting $z = [y_1,\ldots,y_{n-1},w,y_n]$, we get
 \begin{equation}\label{eq:liftedlin}\dot z = [A[z_1,\ldots z_n]^\top,z_n+\sum_{i=1}^{n-1} \sum_{j=1}^{n} \dfrac{\partial g}{\partial z_i} A_{ij}z_j]^\top.\end{equation}
Using arguments similar to the ones used in the proof of Lemma~\ref{lemma:2}, we can show that system~\eqref{eq:liftedlin} $\in \cW_{n+1}$
\end{proof}

\section{Conclusion and Future Work}

We proved that the set of super-linearizable polynomial vector fields is closed under both tame and stably tame automorphisms. However, the case of systems satisfying the WDG condition---a sufficient condition for super-linearization---is more subtle. We presented an example showing that, in general, the WDG condition is not preserved by elementary automorphisms. Nevertheless, by lifting the system using a single stabilizing observable—whose expression can be {\em explicitly obtained} from the automorphism---a linear system can be transformed into one that satisfies the WDG condition.

This result raises several interesting open questions. For instance, given a polynomial system satisfying the WDG condition, can we always add a {\em few} stabilizing observables to the system obtained after applying an automorphism to maintain the WDG condition? Notably, without restricting to a few observables, this follows directly from our results: since the system is super-linearizable by Theorem~\ref{thm:tame}, the transformed system is also super-linearizable, meaning it can be lifted to a linear system, which naturally satisfies the WDG condition.

Another important direction for future work is to explore whether the WDG condition can be generalized to make it {\em invariant} under automorphisms, thereby removing the need for stabilizing variables. In particular, this leads to the following key question: given a super-linearizable system, is there an automorphism that transforms it into a system that satisfies the WDG condition?


\appendix

\section{Proof of Lemma 2}\label{subsec:linearoperatorproof}
\begin{proof}
    We know that $\dot x = f(x)$ is superlinearizable if and only if $$\operatorname{span}\{f ,{\mathcal L}_f f,\ldots, \cL_f^k f,\ldots \}$$ is finite dimensional (see Section~\ref{sec:background}).
    Let $z:=Px$,  $\tilde f (z) := Pf(P^{-1}z)$ and consider $\dot z = \tilde f(z)$. We show, using the criterion mentioned above, that $\tilde f(z)$ is super-linearizable. To this end, we show, by induction on $k$, that $$\dim \Span\{f ,{\mathcal L}_f f,\ldots, \cL_f^k f \}=\dim \Span\{\tilde f ,{\mathcal L}_{\tilde f}\tilde f,\ldots, \cL_{\tilde f}^k \tilde f \}.$$

\xc{Base case:} A direct calculation gives \begin{equation*}
    \tilde f(z)= Pf(x),
\end{equation*}
which shows that $\operatorname{span}\{\tilde f\} = P \operatorname{span} \{f\}$.

\xc{Inductive step:} We claim that for all $k \geq 0$, 
$$\cL_{\tilde f}^k \tilde f = P \cL_f^k f.$$
Let $k\geq 1$ and assume the claim is true for all $ 0 \leq \ell \leq k$.
For $k+1$, we have 

\begin{align*}
[{\mathcal L}_{\tilde f}^{k+1} \tilde f]_i &= \sum_{j=1}^n\dfrac{\partial_i {\mathcal L}_{\tilde f}^{k} \tilde f}{\partial  z_j}  \tilde f_j = \sum_{j=1}^n \frac{\partial_i (P {\mathcal L}_{f}^k f)}{\partial  z_j}  \tilde f_j
\\ &= P \sum_{j=1}^n \dfrac{\partial_i ( {\mathcal L}_{ f}^k f)}{\partial  z_j}   \sum_{m=1}^n P_{jm} f_m \\ 
& = P \sum_{j=1}^n \sum_{m=1}^n \sum_{q=1}^n \dfrac{\partial_i ( {\mathcal L}_{ f}^k f)}{\partial  x_q} P^{-1}_{qj} P_{jm} f_m\\
& = P \sum_{j=1}^n \sum_{m=1}^n \sum_{q=1}^n \dfrac{\partial_i ( {\mathcal L}_{ f}^k f)}{\partial  x_q} \delta_{qm} f_m  
\end{align*}
\begin{align*}
&= P \sum_{j=1}^n \sum_{q=1}^n \dfrac{\partial_i ( {\mathcal L}_{ f}^k f)}{\partial  x_q} f_q \\
& =  P[ {\mathcal L}_{ f}^{k+1}  f ]_i =  [P {\mathcal L}_{ f}^{k+1}  f ]_i
\end{align*}
So, we have $\cL_{\tilde f}^{k+1} \tilde f(z)  = P {\mathcal L}_{ f}^{k+1}  f(x) $ as claimed.  We conclude that  $\operatorname{span}\{\tilde f ,{\mathcal L}_{\tilde f} \tilde f,\ldots \} =  \operatorname{span}\{Pf ,P{\mathcal L}_f f,\ldots \}$ and thus their dimensions are the same. This proves the inductive step and the Lemma.
\end{proof}

\section{Non-superlinearizability of Example~\ref{eq:example_sinh}}\label{subsec:ex1_proof}
\begin{proof}
To prove  that \begin{equation}\label{eq:sysnonsup}
\dot z =  \begin{bmatrix} \sqrt{1+z_1^2}\sinh^{-1}{z_2}\\  \sqrt{1+z_2^2}\sinh^{-1}{z_1}\end{bmatrix}=:f(z)
\end{equation}
is not super-linearizable, we will argue that  $\dim \Span\{f,\cL_f f,\ldots \} = +\infty$. 

To this end, let 
\begin{eqnarray*}q_1 &= \sinh^{-1}{z_1},\qquad  
q_2 &= \sinh^{-1}{z_2}, \\ r_1 &= \sqrt{1+z_1^2}, \qquad r_2 &= \sqrt{1+z_2^2}.
\end{eqnarray*}
 
Note that $\sinh^{-1}{z_2} =q_2$ is not a polynomial in $z_2$ and that  $r_2$ is not a polynomial in $z_2$. Further, $q_1$ cannot be expressed as a polynomial in $r_1,r_2,z_1,z_2$.

\noindent{\textbf{Claim 1:}} $q_1,q_1^2,\ldots q_1^n$ are linearly independent.

\begin{proof}
 Suppose that there exists $\alpha_1, \alpha_2, \ldots, \alpha_n \in \R$ such that $$\alpha_1 q_1(z_1) + \alpha_2q_1(z_1)^2 +\ldots \alpha_nq_1(z_1)^n = 0$$ for all  $z_1\in \R$.
In particular, the above expression for $z_1 = 1,2,\ldots ,n$ yield the following linear system for the coefficients $\alpha_i$:
\begin{align}\label{eq:sysvandermonde}
 \underbrace{\left[\begin{array}{c c c c c} 
    q_1(1) & q_1(1)^2 & q_1(1)^3 & \cdots & q_1(1)^n \\
    q_1(2) & q_1(2)^2 & q_1(2)^3 & \cdots & q_1(2)^n \\
    \vdots & \vdots & \vdots & \ddots & \vdots \\
    q_1(n) & q_1(n)^2 & q_1(n)^3 & \cdots & q_1(n)^n 
\end{array}\right]}_{V} \left[\begin{array}{c}
\alpha_1 \\ \alpha_2  \\ \vdots \\ \alpha_n\end{array}\right] =\left[\begin{array}{c}
0 \\ 0 \\ \vdots \\ 0\end{array}\right]
\end{align}
Note that $q_1\neq 0$ for all $z_1$ chosen above. The matrix $V$ on the left of the above equation is a Vandermonde matrix, whose determinant is known to be equal to  $$  \det(V)=\prod\limits_{1\leq i<j\leq n}(q_1(j)-q_1(i)).$$
Now, since the derivative of $\sinh(z)$ never vanishes,  $q_1(i)^m \neq q_1(j)^m$ for all $ m\in \N$ and $ i\neq j$.  We thus have that $\det(V) \neq 0$. Hence, system~\eqref{eq:sysvandermonde} shows that $\alpha_i = 0$, for $1 \leq i \leq n$ which proves the claim.
\end{proof}



Recalling that system~\eqref{eq:sysnonsup} was obtained by changing variables for $\dot x = [x_2, x_1]^\top$ with $z_1=\sinh x_2$ and $z_2= \sinh x_1$, we have that
$$\left\{\begin{aligned}\dot q_1 &= q_2\\ \dot q_2&= q_1.\end{aligned}\right.
$$
A short calculation show that  $$\left\{\begin{aligned}\dot r_1 &= z_1q_2\\ \dot r_2 &= z_2 q_1.\end{aligned}\right.
$$
\vspace{.3cm}

\mab{Claim 2}: Let $w_1 = p_1(z_1,z_2,r_1,r_2)q_1^n $ where $p_1$ is a polynomial with non-negative coefficients. Then  $$\cL_f w_1 = p_2q_1^{n+1} + p_3q_1^{n} q_2 + p_4q_1^{n-1}q_2$$ where $p_2,p_3,p_4$ are polynomials in $z_1,z_2,r_1,r_2$ with non-negative coefficients. Also, $p_2 \neq 0$ if either  $\frac{\partial p_1}{\partial z_2} \neq 0$ or $\frac{\partial p_1}{\partial r_2} \neq 0$.
\begin{proof}
We can verify the above claim by calculating the total derivative of $w_1$

\begin{align*}
\cL_f w_1 &= \dfrac{\partial w_1}{\partial z_1}\dot z_1 +  \dfrac{\partial w_1}{\partial z_2}\dot z_2 +\dfrac{\partial w_1}{\partial r_1}\dot r_1 +\dfrac{\partial w_1}{\partial r_2}\dot r_2 +\dfrac{\partial w_1}{\partial q_1}\dot q_1    \\
&= q_1^n\dfrac{\partial p_1}{\partial z_1}r_1q_2 + q_1^n\dfrac{\partial p_1}{\partial z_2}r_2q_1 + q_1^n\dfrac{\partial p_1}{\partial r_1}z_1q_2 \\ &+ q_1^n\dfrac{\partial p_1}{\partial r_2}z_2q_1 + p_1q_1^{n-1}q_2\\
& = p_2q_1^{n+1} + p_3q_1^{n} q_2 + p_4q_1^{n-1}q_2
\end{align*}
where $p_2 = \frac{\partial p_1}{\partial z_2}r_2 + \frac{\partial p_1}{\partial r_2}z_2,  p_3 = \frac{\partial p_1}{\partial z_1}r_1+\frac{\partial p_1}{\partial r_1}z_1$ and $p_4= p_1$.

It should be clear that the coefficients of the terms of $p_2,p_3$ and $p_4$ are non-negative if the coefficients of $p_1$ are non-negative. 
Further, also note that if either  $\frac{\partial p_1}{\partial z_2} \neq 0$ or $\frac{\partial p_1}{\partial r_2} \neq 0$,  then $p_2 = \frac{\partial p_1}{\partial z_2}r_2 + \frac{\partial p_1}{\partial r_2}z_2 \neq 0$.
\end{proof}

To prove that $\Span\{f,\cL_f f,\ldots \}$,  is not finite-dimensional, it is sufficient to show that these successive derivatives can be written as polynomials in $q_i$ whose degree is unbounded. We do so next:

\mab{Claim 3}: The term with  highest degree in $q_1$ of $[\cL_f^k f]_2$ is of the form $p(z_1,z_2,r_1,r_2)q_1^{k+1}$ where $p$ is a polynomial with non-negative coefficients in $z_1,z_2,r_1,r_2$ and at least one of $\dfrac{\partial p}{\partial z_2}\neq 0$, $\dfrac{\partial p}{\partial r_2} \neq 0$ holds.

\begin{proof}
We will prove the claim by induction. 

\xc{Base case $k=1$:} We have

$${\mathcal L}_f f = \left[\begin{array}{c} 
 z_1 q_2^2 + r_1q_1\\
 z_2 q_1^2  + r_2q_2
 \end{array}\right].$$
We see that the leading term in $[\cL_f f]_2$ is $z_2q_1^2$, which is of the required form.

\xc{Inductive step:} We assume that the claim is true for all $1\leq \ell \leq  k$. 
The leading term of $[\cL_f^k f]_2$ is thus of the form $w =p_1(z_1,z_2,r_1,r_2)q_1^{k+1}$ where $p$ is a polynomial in $z_1,z_2,r_1,r_2$ with non-negative coefficients and $\frac{\partial p}{\partial z_2}\neq 0$ or $\frac{\partial p}{\partial r_2} \neq 0$.

Given $w$ as above, the degree of $\cL_f w$ in $q_1$ is $k+2$ from claim 2:
$$\cL_f w = p_2q_1^{k+2} + p_3q_1^{k+1} q_2 + p_4q_1^{k}q_2
$$
 where $p_2,p_3,p_4$ are polynomials with non-negative coefficients, and \begin{equation}\label{eq:defp2}p_2 = \dfrac{\partial p_1}{\partial z_2}r_2 + \dfrac{\partial p_1}{\partial r_2}z_2 \neq 0.
 \end{equation}
 From the previous equation,  we obtain that either 
 
 \begin{equation}\label{eq:alternative}\dfrac{\partial p_1}{\partial z_2} \neq 0 \mbox{ and/or } \dfrac{\partial p_1}{\partial r_2} \neq 0
 \end{equation}

Differentiating~\eqref{eq:defp2}, we obtain 
\begin{equation}\label{eq:partialp2}
\left\{ \begin{aligned}
    \dfrac{\partial p_2}{\partial r_2} &= \left(\dfrac{\partial p_1}{\partial z_2} + \dfrac{\partial^2 p_1}{\partial r_2\partial z_2}r_2\right) + (\dfrac{\partial^2 p_1}{\partial r_2^2}z_2) \\
     \dfrac{\partial p_2}{\partial z_2} &= \left(\dfrac{\partial^2 p_1}{\partial z_2^2}r_2 \right) + (\dfrac{\partial^2 p_1}{\partial r_2\partial z_2}z_2 + \dfrac{\partial p_1}{\partial r_2})   
\end{aligned}\right.
\end{equation}

We claim that at least one of $\frac{\partial p_2}{\partial r_2}\neq 0$ and   $\frac{\partial p_2}{\partial z_2}\neq 0$ holds. To verify the claim, assume, to the contrary that both the derivatives vanish. Then, since $p_1$ has non-negative coefficients by assumption, we have from~\eqref{eq:partialp2} that both $\frac{\partial p_1}{\partial z_2}=0$ and $\frac{\partial p_1}{\partial r_2} =0  $, which contradicts~\eqref{eq:alternative} and proves the claim.

Hence, the leading term of $[\cL_f^{k+1} f]_2$ is of the required form---this concludes the proof of the claim.
\end{proof}
Now, since the exponent of the leading term of $[\cL_f^k f]_2$ increases with $k$ as shown in claim 3 and  $q_1,q_1^2,\ldots$ are linearly independent by claim 1 and can not be obtained by any linear combination of polynomials in $z_1,z_2,r_1,r_2$, we have that $$\cL_f^{k+1} f \notin \Span\{f,\cL_f f,\ldots \cL_f^{k} f\} \mbox{ for all } k\in \N.$$ Hence, $\dim \Span\{f,\cL_f f,\ldots\}  = +\infty$ .
\end{proof}

\bibliographystyle{amsplain}
\bibliography{refs.bib}

\end{document}